\documentclass{amsart}
\usepackage{amsfonts,  mathtools, bbm, array}
\usepackage{graphicx}
\usepackage{overpic}
\usepackage{epsfig} 
\usepackage{hyperref}
\usepackage{MnSymbol}  
\usepackage{color}

\newtheorem{theorem}{Theorem}[section]
\theoremstyle{plain}
\newtheorem{corollary}[theorem]{Corollary}
\newtheorem{lemma}[theorem]{Lemma}
\newtheorem{proposition}[theorem]{Proposition}
\theoremstyle{remark}

\newtheorem{example}[theorem]{Example}
\numberwithin{equation}{section}

\newcommand{\tr}{\operatorname{tr}}

\newcommand{\calC}{{\mathcal C}}

\newcommand{\calP}{{\mathcal P}}

\newcommand{\del}{\partial}

\newtheorem*{prop*}{Theorem}

\newtheorem{defi}[theorem]{Definition} 
\newtheorem{prop}[theorem]{Proposition}

\newcommand{\bbR}{\mathbb{R}}

\newcommand{\bbZ}{\mathbb{Z}}
\newcommand{\bbN}{\mathbb{N}}
\newcommand{\norm}[1]{\left\Vert #1 \right\Vert}
\newcommand{\abs}[1]{\left|#1 \right|}

\newcommand{\paths}{\calP}

\def\ins#1#2#3{\vbox to0pt{\kern-#2 \hbox{\kern#1 #3}\vss}\nointerlineskip}

\begin{document}

\title{The heat kernel on the diagonal  for a compact metric graph}
\author[Borthwick]{David Borthwick}
\address{Department of Mathematics, Emory University, Atlanta, GA 30322}
\email{dborthw@emory.edu}
\author[Harrell]{Evans M. Harrell II}
\address{School of Mathematics, Georgia Institute of Technology, Atlanta, GA 30332}
\email{harrell@math.gatech.edu}
\author[Jones]{Kenny Jones}
\address{Department of Mathematics, Emory University, Atlanta, GA 30322}
\email{wesley.kenderdine.jones@emory.edu}
\date{\today}


\begin{abstract}
We analyze the heat kernel associated to the Laplacian on a compact metric graph, with standard Kirchoff-Neumann vertex conditions.   An explicit formula for the heat kernel as a sum over loops, developed by Roth and Kostrykin-Potthoff-Schrader,
allows for a straightforward analysis of small-time asymptotics.  
We show that the restriction of the heat kernel to the diagonal satisfies a modified version of the heat equation.  
This observation leads to an ``edge'' heat trace formula, expressing the a sum over eigenfunction amplitudes on a single edge as a sum over closed loops containing that edge.  The proof of this formula relies on a modified heat equation satisfied by the diagonal restriction of the heat kernel.  Further study of this equation leads to explicit formulas for completely symmetric graphs.
\end{abstract}

\maketitle

\section{Introduction}

Let $G$ be a compact, connected metric graph.  
The Laplacian operator $-\Delta$ on $G$
is the self-adjoint operator on $L^2(G)$ associated to the
quadratic form $\norm{u'}^2$ with domain $H^1(G)$.  
On each edge $-\Delta$ acts as the differential operator $-d^2/dx^2$, subject to the standard
Kirchhoff-Neumann vertex conditions, which require that the outward derivatives at each vertex sum to zero.

The paper is devoted to the study of the integral kernel of the heat operator $e^{t\Delta}$, which we denote by
$H(t,\cdot,\cdot)$. Most of our results concern the restriction of heat kernel to the diagonal, 
\[
h(t,q) := H(t,q,q).
\]
We associate to $-\Delta$ an orthonormal basis of real-valued eigenfunctions $\psi_j$, 
with corresponding eigenvalues $\lambda_j$. 
Since $G$ is connected, $\lambda_1 = 0$ is a simple eigenvalue. 
The eigenfunctions $\psi_j$ for $j >0$ are generally not uniquely determined, because of multiplicities
in the spectrum.

In terms of the eigenvalues 
and eigenfunctions, the heat kernel admits the expansion
\begin{equation}\label{hk.eigf}
H(t,q_1,q_2) = \sum_{j=1}^\infty e^{-t\lambda_j} \psi_j(q_1) \psi_j(q_2),
\end{equation}
which converges uniformly on $G \times G$ for $t>0$.

\begin{example}\label{star.ex}
Let $G$ be a star graph with $d$ equal edges of length $a$.  Eigenfunctions which are nonzero at the central vertex must be symmetric,
and thus proportional to $\cos(\tfrac{\pi k}{a} x)$ for $k \in \bbN_0$, with $x$ the coordinate measured outward from the center on each edge.  

The eigenfunctions which vanish at the center are given by 
$c_j \sin (\tfrac{\pi}{a}(k + \tfrac12) x)$ for $k \in \bbN_0$, where $c_j$ is the coefficient 
for edge $j$.  The vertex condition implies $\sum c_j = 0$, so this yields a $(d-1)$-dimensional eigenspace for each $k$.   

Converting the eigenfunctions into an orthonormal basis yields the expansion, for the restriction to the diagonal,
\begin{equation}\label{htx.star}
\begin{split}
h(t,x) &= \frac{1}{ad} + \frac{2}{ad} \sum_{k=1}^\infty e^{-(\frac{\pi k}{a})^2t} \cos^2(\tfrac{\pi k}{a} x) \\
&\qquad + \frac{2(d-1)}{ad} \sum_{k=0}^\infty e^{-(\frac{\pi}{a}(k+\frac12))^2t} \sin^2(\tfrac{\pi}{a}(k+\tfrac12) x).
\end{split}
\end{equation}
Observe that the value of $h(t,\cdot)$ at the central vertex is $1/d$ times the value at an endpoint of a Neumann interval of length $a$.
\end{example}

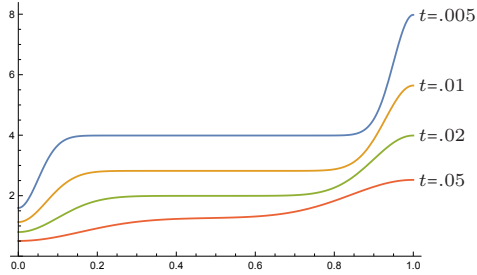
\begin{figure}
\begin{center}
\begin{overpic}[scale=.6]{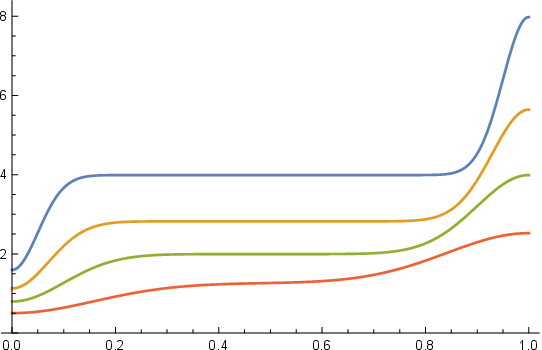}
\put(99,60){$\scriptstyle t = .005$}
\put(99,43){$\scriptstyle t = .01$}
\put(99,31){$\scriptstyle t = .02$}
\put(99,20){$\scriptstyle t = .05$}
\end{overpic}
\end{center}
\caption{The function $h(t,\cdot)$ on an edge of the star graph with $a=1$, $d=5$.}\label{star_htx.fig}
\end{figure}

The small-time behavior of $h(t,\cdot)$ near vertices is illustrated in Fig.~\ref{star_htx.fig}.  Near a vertex of degree $d$, as $t \to 0$ 
the diagonal heat kernel approaches a value $2/d$ times its value at a generic edge point.  A precise version of this statement 
can be derived from the heat kernel formulas of Roth \cite{Roth:1984} and Kostrykin-Potthoff-Schrader \cite{KPS:2007}; see Proposition~\ref{KPS.prop}.
These asymptotics yield a local Weyl law (Theorem~\ref{local.weyl.thm}):
\[
\lim_{N \to \infty} \frac{1}{N} \sum_{j=1}^N \psi_j(q)^2 = \frac{2}{d_q L},
\]
where $L$ is the total length of $G$ and $d_q$ is the degree of $q$, interpreted as $2$ if $q$ is an interior edge point.

We can also use the analysis of $h(t,\cdot)$ to study the average concentration of eigenfunctions on an edge.
To formulate this result, note that on an edge ${\bf e}$ of $G$ parametrized by $x \in [0,a]$, 
each (normalized) eigenfunction takes the form 
\[
\psi_j(x) = b_j({\bf e}) \cos( \sigma_j x + \phi_j),
\]
where $\lambda_j = \sigma_j^2$.  The phases can be adjusted so that $b_j>0$, 
which makes $b_j({\bf e})$ uniquely determined by $\psi_j$.
In \S\ref{eigf.sec}, we prove an edge version of the Weyl asymptotic, 
\[
\lim_{N \to \infty} \frac{1}{N} \sum_{j=1}^N b_j({\bf e})^2 = \frac{2}{L}
\]
on each edge.

The amplitudes $b_j$ can also be used to prove an interesting variant of the heat trace formula.  A trace formula expressing 
$\tr e^{t\Delta}$ as a sum over paths on $G$ was established by Roth \cite{Roth:1984}.  Theorem~\ref{edge.tr.thm} gives
an ``edge'' trace formula:  for some constant $c_{\bf e}$ depending on the edge ${\bf e}$,
\[
\frac12 \sum_{j=1}^\infty e^{-\lambda_j t} b_j({\bf e})^2 =  \frac{1}{\sqrt{4\pi t}} \sum_{\gamma \in \calP_{\bf e}} 
\alpha(\gamma) e^{-\ell(\gamma)^2/4t} + c_{\bf e},
\]
where $\calP_{\bf e}$ is the collection of closed paths (including the trivial path) which start with a bond in ${\bf e}$.

Finally, we consider a modified heat equation satisfied by $h(t,q)$, which can be solved under certain
conditions.   We apply this approach to work out  explicit formulas for $h(t,q)$ for graphs with a high degree of symmetry.

\vskip12pt\noindent
\textbf{Acknowledgment.}  We are grateful to Livia Corsi for her participation in the early discussions from which this project developed.
Comments and corrections from the anonymous reviewers are much appreciated.
The figures in this paper were produced using Mathematica 13.0.

\section{Preliminaries}\label{prelim.sec}

An expansion of the heat kernel on a metric graph as a sum over paths was first developed for compact graphs with Kirchoff-Neumann vertex conditions
by Roth \cite[\S III]{Roth:1984}.  This was later generalized to infinite graphs in Cattaneo \cite[Thm.~1]{Cattaneo:1999}, 
and to more general vertex conditions in Kostrykin-Potthoff-Schrader \cite[Cor.~3.4]{KPS:2007}. An extension to 
more general convolution semigroups was given in Becker-Gregorio-Mugnolo \cite[Thm.~1]{BGM:2021}

Our attention is restricted to the heat kernel of the Laplacian on a compact metric graph $G$ with Kirchhoff-Neumann vertex conditions.
 The formula from \cite{KPS:2007} does not allow loops (tadpoles), where a single edge is attached to the same vertex at both ends.  
However, in the case of Kirchhoff-Neumann vertex conditions we could work around this issue by inserting an artificial (degree two) 
vertex into each loop.  We will instead follow the approach from \cite{Roth:1984}, which is to describe paths in terms of 
oriented edges. 

In the terminology of \cite{BK:2013}, an oriented edge of $G$ is called a \emph{bond}, so that each edge corresponds to exactly two bonds. 
(The corresponding term in the discrete graph literature is \emph{arc}.)
For each bond $\vec{e}$ we can identify an initial vertex $\del^-(\vec{e})$ and final vertex $\del^+(\vec{e})$.  Note that a loop
attached at vertex $v$ is associated to two bonds, both of which have $v$ as initial and final vertex.

Two bonds $\vec{e}_1$ and $\vec{e}_2$ are \emph{consecutive} if $\del^+(\vec{e}_1) = \del^-(\vec{e}_2)$.
A \textit{path} $\gamma$ consists of a pair of vertices connected by an ordered sequence of consecutive bonds, i.e.,
\begin{equation}\label{path.def}
\gamma = (v_-,\vec{e}_1, \dots, \vec{e}_n, v_+)
\end{equation}
where $v_- = \del^-(\vec{e}_1)$ and $v_+ = \del^+(\vec{e}_n)$.  
A path may ``bounce'' at a vertex, i.e., a bond may be followed by its inverse bond.  
Also, the trivial path connecting a vertex to itself is allowed
and denoted by $\gamma = (v,v)$.  The trivial path is not assigned a direction and hence
counted only once.
For any path $\gamma$ we denote the total length by 
\[
\ell(\gamma) := \sum_{j=1}^n \ell(\vec{e}_j)
\]

To formulate the expansion for the heat kernel $H(t,\cdot,\cdot)$ in a way that includes vertex points, it is convenient to introduce 
artificial vertices at the evaluation points, if needed.  Thus, in the formula for $H(t,q_1,q_2)$ we will assume that 
$q_1$ and $q_2$ are vertices, possibly of degree two if the original points were interior to an edge.
With this convention,  let $\paths(q_1, q_2)$ denote the collection of paths with $v_- = q_1, v_+ = q_2$.  
This includes the trivial path if $q_1 = q_2$.

To each path $\gamma$ we assign a coefficient $\alpha(\gamma)$ defined as follows.  For the trivial path,
\begin{equation}\label{alpha.triv.def}
\alpha((v,v)) := \frac{2}{\deg(v)}.
\end{equation}
For a path with at least one edge,
\begin{equation}\label{alpha.def}
\alpha((v_-,\vec{e}_1, \dots, \vec{e}_n, v_+)) := \frac{4}{\deg(v_-)\deg(v_+)} \prod_{j=1}^{n-1} \beta(\vec{e}_j, \vec{e}_{j+1}),
\end{equation}
where $\beta$ is the bond-scattering matrix element, defined as 
\[
\beta(\vec{e}_j, \vec{e}_{j+1}) := 
\begin{dcases} 
\frac{2}{\deg \del^+(\vec{e}_j)}, &\del^-(\vec{e}_j) \ne \del^+(\vec{e}_{j+1}) \ \text{(transfer)},\\
\frac{2}{\deg \del^+(\vec{e}_j)} - 1, &  \del^-(\vec{e}_j) = \del^+(\vec{e}_{j+1}) \ \text{(bounce)}.
\end{dcases}
\]
Note that an artificial vertex contributes a factor of 0 for a bounce and 1 for a transfer.  Thus artificial vertices are essentially
invisible except as potential terminal points for a path.
Because of this, we may omit from $\calP(q_1,q_2)$ any paths that bounce at artificial vertices.  

We are now prepared to state the {\em path sum formula} for the heat kernel. The versions of this formula from \cite{Roth:1984} and 
\cite{KPS:2007} were restricted to points interior to the edges. 
The novel feature here is that the definitions of $\calP(q_1,q_2)$ and of the path coefficient
$\alpha$ have been adapted to account for general points on the graph.
\begin{proposition}\label{KPS.prop}
For points $q_1,q_2 \in G \times G$ (including vertices) the heat kernel for 
the Kirchhoff-Neumann Laplacian has the expansion
\begin{equation}\label{KPS.formula}
H(t,q_1,q_2) = \frac{1}{\sqrt{4\pi t}} \sum_{\gamma \in \paths(q_1,q_2) }\alpha(\gamma) e^{-\ell(\gamma)^2/4t},
\end{equation}
which converges uniformly in $G \times G$ for all $t >0$.
\end{proposition}

\begin{proof}
For interior edge points, \eqref{KPS.formula} follows from either \cite[\S III]{Roth:1984} or 
\cite[Cor.~3.4]{KPS:2007}. Our goal is thus to check the cases where either $q_j$ could be a
vertex.

The behavior at vertex points can derived from the general formula by continuity.  Suppose that $q_1$ is interior to an edge 
with vertex $v_0$ (of degree $\ge 3$), and assume for the moment that $q_2$ is fixed and not equal to $v_0$,  We will 
set $x = d(v_0,q_1)$ and consider the limit $x \to 0$.  Paths in $\calP(q_1,q_2)$ may be divided into two types.  A path $\gamma'$
of the first type consists of paths $\gamma$ whose first step is a transfer at $v_0$, as illustrated in Figure~\ref{type1.fig}.
Truncating the first bond of $\gamma'$ gives a corresponding path 
$\gamma \in \calP(v_0,q_2)$, such that $\ell(\gamma') = \ell(\gamma) + x$.  
By \eqref{alpha.def} the path coefficients are equal,
\[
\alpha(\gamma) = \alpha(\gamma'),
\]
because the transfer factor of $2/\deg(v_0)$ from $\alpha(\gamma')$ appears as the 
$2/\deg(v_-)$ in the expression for $\alpha(\gamma)$.  
We thus have
\[
\lim_{x\to 0} \alpha(\gamma') e^{-\ell(\gamma')^2/4t} =  \alpha(\gamma) e^{-\ell(\gamma)^2/4t}.
\]

\begin{figure}
\begin{tabular}{c@{\hspace{.2in}}c}
\begin{overpic}[scale=.5]{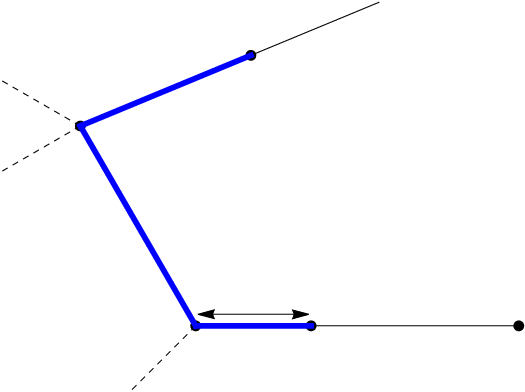}
\put(36,8){$v_0$}
\put(58,8){$q_1$}
\put(98,8){$v_1$}
\put(45,68){$q_2$}
\put(47,16){$x$}
\put(46,37){$\gamma'$}
\end{overpic} &
\begin{overpic}[scale=.5]{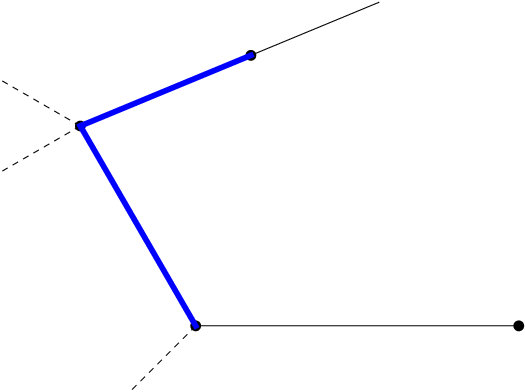}
\put(36,8){$v_0$}
\put(98,8){$v_1$}
\put(45,68){$q_2$}
\put(46,37){$\gamma$}
\end{overpic}
\end{tabular}
\caption{A path $\gamma' \in \calP(q_1,q_2)$ of the first type, and its truncation to $\gamma \in \calP(v_0,q_2)$.}\label{type1.fig}
\end{figure}
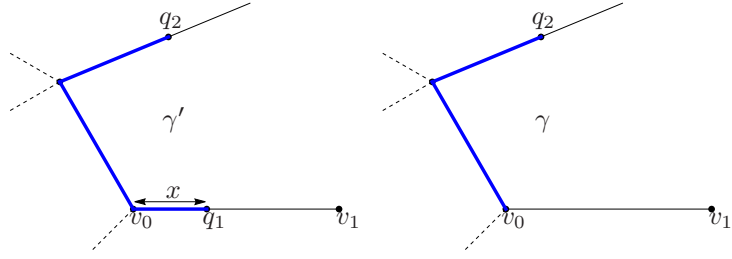

Paths of the second type either bounce at $v_0$ or first pass through $v_1$.  
These paths can be paired.
For each path $\gamma' \in \calP(q_1,q_2)$ that does not initially pass through $v_0$, 
there is a partner $\gamma''$ which first bounces at $v_0$, passes
through $q_1$, and then follows the same route thereafter.  Associated to this pair is a unique path $\gamma \in \calP(v_0,q_2)$ which 
first hits $v_1$.  The lengths of these paths are related by
\[
\ell(\gamma) = \ell(\gamma') + x = \ell(\gamma'') - x,
\]
and the coefficients by
\[
\alpha(\gamma) = \frac{2}{\deg(v_0)} \alpha(\gamma'), \qquad  \alpha(\gamma'') = \left( \frac{2}{\deg(v_0)} - 1\right) \alpha(\gamma').
\]
The limit of the contributions from to the heat kernel from $\gamma'$ and $\gamma''$ are thus accounted for by
\[
\begin{split}
\lim_{x\to 0} \left[ \alpha(\gamma')e^{-\ell(\gamma')^2/4t} + \alpha(\gamma'')e^{-\ell(\gamma'')^2/4t} \right] 
&=  \frac{2}{\deg(v_0)} \alpha(\gamma')e^{-\ell(\gamma)^2/4t} \\
&= \alpha(\gamma)e^{-\ell(\gamma)^2/4t}.
\end{split}
\]
Similar considerations apply to the off-diagonal case where $q_2$ approaches a vertex.

\begin{figure}
\begin{tabular}{c@{\hspace{.5in}}c}
\begin{overpic}[scale=.5]{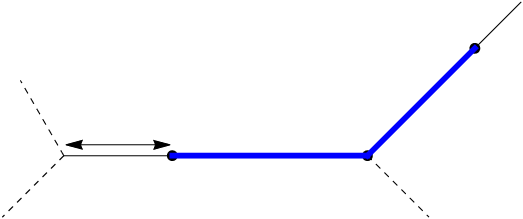}
\put(11,7){$v_0$}
\put(31,7){$q_1$}
\put(67,7){$v_1$}
\put(90,28){$q_2$}
\put(21,15){$x$}
\put(50,25){$\gamma'$}
\end{overpic} 
&
\begin{overpic}[scale=.5]{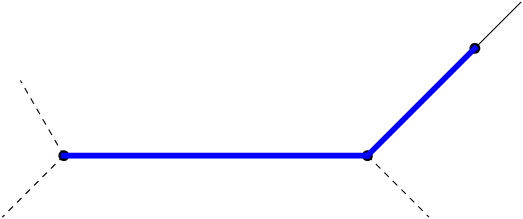}
\put(11,6){$v_0$}
\put(67,7){$v_1$}
\put(90,28){$q_2$}
\put(50,25){$\gamma$}
\end{overpic} \\
& \\
\begin{overpic}[scale=.5]{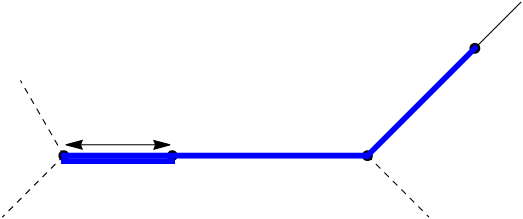}
\put(11,6){$v_0$}
\put(31,6){$q_1$}
\put(67,7){$v_1$}
\put(90,28){$q_2$}
\put(21,15){$x$}
\put(50,25){$\gamma''$}
\end{overpic} &
\end{tabular}
\caption{A pair of paths $\gamma',\gamma'' \in \calP(q_1,q_2)$ of the second type, and the corresponding
$\gamma \in \calP(v_0,q_2)$.}\label{type2.fig}
\end{figure}
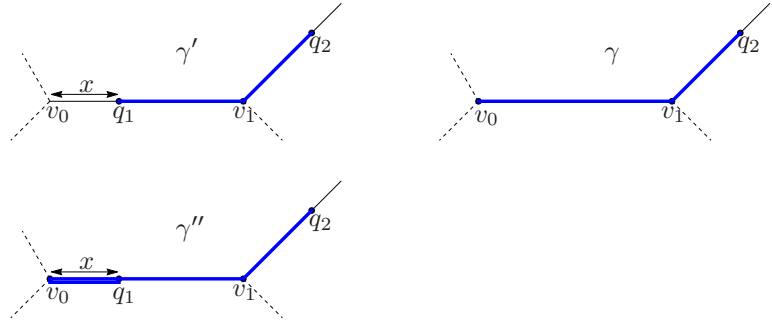

For the diagonal case, assume that $q$ is an interior edge point approaching a vertex $v_0$, with $x = d(v_0,q)$.  
The trivial path $(v_0,v_0)$
is the limit of two paths in $\calP(q,q)$, the trivial path $(q,q)$ and the path of length 
$2x$ which bounces off $v_0$.  The limit of these two terms is given by
\[
\lim_{x\to 0} \left[ 1 + \left( \frac{2}{\deg(v_0)}-1\right) e^{-x^2/t} \right] =  \frac{2}{\deg(v_0)},
\]
which agrees with $\alpha((v_0,v_0))$.

Finally, we consider a non-trivial path $\gamma \in \calP(v_0,v_0)$, as the limit of paths in $\calP(q,q)$.  
If neither the initial nor final bonds of $\gamma$ pass through $q$, 
then there is only one corresponding path $\gamma' \in \calP(v_0,v_0)$,
which transfers through $v_0$ in its first and last steps and follows $\gamma$ in between.  
This is similar to the first type shown in Figure~\ref{type1.fig}, except that $q_2 = q_1$
and $\ell(\gamma') = \ell(\gamma) + 2x$.
The initial and final transfers of $\gamma'$ contribute a factor of
$(2/\deg(v_0))^2$ to $\alpha(\gamma')$.  
These factors are included in $\alpha(\gamma)$ by the definition \eqref{alpha.def}, so that
\[
\alpha(\gamma) = \alpha(\gamma').
\]
This accounts for the limit when $q$ does not lie on the initial or final edges of $\gamma$.  
If $q$ does lie on one or both of these edges, then the same factor occurs.  
This can be derived, as in the case shown in Figure~\ref{type2.fig}, 
by combining paths with a short bounce at $v_0$ with corresponding paths that terminate
at $q$ directly. 

The argument for uniform convergence was given in \cite[Cor.~3.4]{KPS:2007}.
Let $a_0$ denote the minimum edge length of $G$.  The set $\paths(q_1,q_2)$ contains at most 
3 paths which do not contain a true edge (of the original graph without artificial vertices).  
If $G$ has $m$ edges, then the number of paths that contain $k$ full edges is bounded by $m^k$.
The contribution to the sum from the paths that contain at least one true edge is bounded by
\[
\sum_{k=1}^\infty m^k e^{-k^2a_0^2/4t} < \infty,
\]
independently of the choice of points.
\end{proof}

\begin{example}
For the interval $[0,a]$, with Neumann boundary conditions, the eigenfunction expansion of the heat kernel is
\[
H(t,x,y) = \frac{1}{a} + \frac{2}{a} \sum_{j=1}^\infty e^{-(\frac{j\pi}{a})^2t} \cos(\tfrac{j \pi }{a}x) \cos(\tfrac{j \pi}{a}y).
\]
The path sum formula of  Proposition~\ref{KPS.prop} yields the same result as the method of images:
\begin{equation}\label{interval}
H(t,x,y) = \frac{1}{\sqrt{4\pi t}} \sum_{k\in \bbZ} \left[ e^{-(x-y+2ka)^2/4t} + e^{-(x+y+2ka)^2/4t} \right].
\end{equation}
Restricting to the diagonal gives
\[
h(t,x) =  \frac{1}{\sqrt{4\pi t}} \left[ 2 + 2 \sum_{k=1}^\infty e^{-(ka)^2/t} + 2 \sum_{k=1}^\infty e^{-(x+ka)^2/t} \right],
\]
For $h(t,0)$ the the sums combine to give a total coefficient of $4$, which shows the factor $(2/\deg)^2$ from \eqref{alpha.def}.
\end{example}

\begin{example}
Consider a two-petal flower graph (i.e., two loops joined at a point) of total length $L$.
The graph is invariant under reflection of the petals keeping the vertex and the extreme point of the petals invariant. 
We can thus assume that the eigenfunctions are either even or odd with respect to the reflection.  
Suppose that the edge ${\bf e}_1$ has length $a<L$ and is parametrized by $x \in [0,a]$.  
Then the even eigenfunctions are proportional to $\cos\left( \frac{2\pi k}{L}(x - a/2)\right)$ while the odd eigenfunctions vanish at the vertex.  
Consequently, the eigenfunction expansion for the heat kernel on the diagonal on this edge, denoted by 
$h_{{\bf e}_1}(t,\cdot)$, can be written in the form
\[
\begin{split}
h_{{\bf e}_1}(t,x) &= \frac{1}{L} + \frac{2}{L} \sum_{k=1}^\infty e^{- (2\pi k/L)^2 t} \cos^2\left( \frac{2\pi k}{L}(x - a/2) \right) \\
&\qquad + \frac{2}{a} \sum_{k=1}^\infty e^{-(2\pi k/a)^2 t} \sin^2\left( \frac{2\pi k x}{a} \right).
\end{split}
\]
Note that there is no contribution from the eigenfunctions with frequency 
$2\pi k/(L-a)$, because these are supported only on the edge ${\bf e}_2$.

The eigenvalue expansion of $h_{{\bf e}_1}(t,\cdot)$ can be rewritten in the form of \eqref{KPS.formula} using Poisson summation.  The result is
\[
\begin{split}
h_{{\bf e}_1}(t,x) &= \frac{1}{\sqrt{4\pi t}} \sum_{n \in \bbZ}  \left[ \frac12 e^{-(nL)^2/4t} + \frac14 e^{-(2x-a+nL)^2/4t}
+ \frac14 e^{-(2x-a-nL)^2/4t} \right] \\
&\qquad + \frac{1}{\sqrt{4\pi t}} \sum_{n\in\bbZ} \left[ \frac12 e^{-(n a)^2/4t} -\frac14 e^{-(2x+na)^2/4t} - \frac14 e^{-(2x-na)^2/4t}\right]
\end{split}
\]
\end{example}

\section{Small-time asymptotics}

Since our primary applications involve eigenvalue concentration, we will focus on restriction of the heat kernel to the diagonal,
denoted by $h(t,\cdot)$.
From Proposition~\ref{KPS.prop}, we derive the following:
\begin{proposition}\label{smallt.prop}
Suppose that the compact metric graph $G$ has $m$ edges in total and  minimum edge length $a_0$.
For a point $q \in G$, let $v_0$ be the nearest (non-artificial) vertex (possibly with $q = v_0$), with $d_0$ the degree of $v_0$.  
Then for any $t_0 < a_0^2/(2 \log m)$, 
\begin{equation}\label{Kqq.smallt}
h(t,q) = \frac{1}{\sqrt{4\pi t}} \left[1 + \left(\frac{2}{d_0}-1\right) 
e^{-d(q,v_0)^2/t} + O(me^{-a_0^2/4t}) \right],
\end{equation}
for $t \in (0,t_0]$, where the constant in the error estimate depends only on $t_0$. 
\end{proposition}
\begin{proof}
Suppose that $q$ lies within the edge ${\bf e}$, with nearest vertex $v_0$.  Let $a$ denote the length of ${\bf e}$, 
so that the edge can be parametrized by $x = d(q,v_0) \in [0,a]$.  The set
$\calP(q,q)$ contains, in addition to the trivial path, paths of length $2x$ and $2(a-x)$.  
All other paths contain at least one complete edge.

In this notation, the restriction of $h(t,q)$ to ${\bf e}$ is given by the  formula \eqref{KPS.formula} as
\[
h_{\bf e}(t,x) = \frac{1}{\sqrt{4\pi t}} \left[1 + \left(\frac{2}{d_0}-1\right) e^{-x^2/t} + \left(\frac{2}{d_1}-1\right)  e^{-(a-x)^2/t} + R(t,x)\right],
\]
where $R(t,x)$ contains the contributions from paths that transfer through $v_0$ or $v_1$.
Suppose that such a path contains $k$ complete edges, not counting the segments that
connect $q$ to a vertex.  
The length of this path is at least $ka_0$, where $a_0$ is the minimum edge length.  
Note that $k \ge 1$ since the path contains a transfer.
The number of paths with $k$ edges is bounded by $m^k$, where $m$ is the total number of edges of $G$.
We can thus estimate
\[
\abs{R(t,x)} \le \sum_{k=1}^\infty m^k e^{-(2x+ka_0)^2/4t}.
\]
For $t < a_0^2/(2\log m)$, we have a linear estimate for $k \ge 1$, 
\[
\frac{(ka_0)^2}{4t} + k \log m \ge -\frac{a_0^2}{4t} + \frac{ka_0^2}{2t} + k \log m.
\]
Applying this to the error term gives
\[
\abs{R(t,x)} \le \sum_{k=1}^\infty m^k e^{a_0^2/4t - ka_0^2/2t} =  \frac{me^{-a_0^2/4t}}{1 - m e^{-a_0^2/2t}}.
\]
\end{proof}

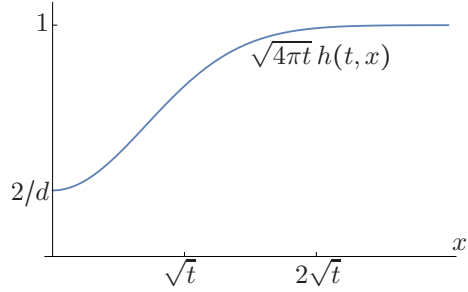
\begin{figure}
\begin{center}
\begin{overpic}[scale=.6]{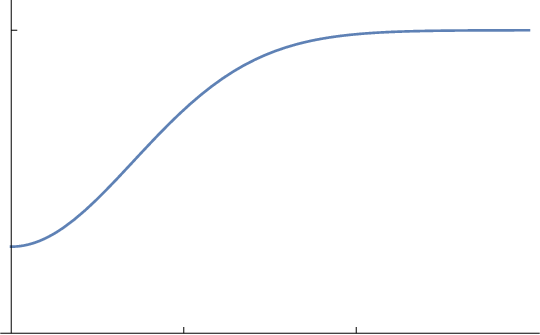}
\put(29,-6){$\sqrt{t}$}
\put(61,-6){$2\sqrt{t}$}
\put(99,2){$x$}
\put(-8,14){$2/d$}
\put(-2,55){$1$}
\put(50,47){$\sqrt{4\pi t}\>h(t,x)$}
\end{overpic}
\end{center}
\caption{Small-time behavior of $h(t,\cdot)$ near a vertex of degree $d$, where $x$ is the distance to the vertex.}\label{hvertex.fig}
\end{figure}

Proposition~\ref{smallt.prop} shows that the leading behavior of $h(t,\cdot)$ as $t\to 0$ depends only on 
the distance to the nearest vertex and the degree of that vertex. 
This is illustrated schematically in Figure~\ref{hvertex.fig}.
Applying a Taylor approximation to \eqref{Kqq.smallt} yields the following:
\begin{corollary}
Let $v_0$ be the vertex nearest $q$ and set $x = d(v_0,q)$.  For $t$ sufficiently small and $x \lesssim \sqrt{t}$,
\[
h(t,q) = \frac{1}{\sqrt{4\pi t}} \left[ \frac{2}{d_0} + \frac{d_0-2}{d_0} \frac{x^2}{t} + O\bigl(\frac{x^4}{t^2}\bigr) + O(e^{-a_0^2/4t}) \right]
\]
\end{corollary}

An analysis similar to that of Proposition~\ref{smallt.prop} is possible in the off-diagonal case, yielding
\[
H(t,q_1,q_2) \sim \frac{c}{\sqrt{4\pi t}} e^{-d(q_1,q_2)^2/4t},
\]
for $q_1 \ne q_2$, where the constant is given by a sum over minimal paths connecting $q_1$ to $q_2$,
\[
c = \sum_{\gamma \in \calP(q_1,q_2):\, \ell(\gamma) = d(q_1,q_2)} \alpha(\gamma).
\]
Minimal paths contain no bounces, so the coefficient $\alpha(\gamma)$ appearing here is the 
product of $2/\deg$ for the vertices along the path.

\section{Eigenfunction concentration}\label{eigf.sec}

The Weyl asymptotic for compact metric graphs is standard
(proofs may be found in \cite{Berk:2017, Harr:2018}) and gives
\begin{equation}\label{weyl.law}
\#\{\lambda_j \le t\} \sim \frac{L}{\pi} t^{1/2},
\end{equation}
where $L$ is the total length of $G$.
From Proposition~\ref{smallt.prop} we can deduce the following local Weyl law:
\begin{theorem}\label{local.weyl.thm}
For a point $q\in G$, the eigenfunctions satisfy, as $N \to \infty$,
\[
\lim_{N\to \infty} \frac{1}{N} \sum_{j=1}^N \psi_j(q)^2 = \frac{2}{d_q L},
\]
where $d_q$ is the degree of $q$, interpreted as $2$ if $q$ is an interior edge point.
\end{theorem}
\begin{proof}
From \eqref{Kqq.smallt} we have the asymptotic
\[
\sum_{j=1}^\infty e^{-t\lambda_j} \psi_j(q)^2 \sim \frac{1}{\sqrt{4\pi t}} \frac{2}{d_q}
\]
as $t \to 0$.  By the Karamata Tauberian theorem (see, for example, \cite[\S6.5.4]{Borthwick2020}), this implies that
\[
\sum_{\lambda_j \le \lambda}  \psi_j(q)^2 \sim \frac{2}{\pi d_q} \lambda^{1/2}
\]
as $\lambda \to \infty$.  The Weyl law \eqref{weyl.law} then allows us to replace $\lambda_j \le \lambda$,
by $j \le N$, where $\lambda \sim (\pi N/L)^2$.
\end{proof}

The result of Theorem~\ref{local.weyl.thm} cannot be directly integrated over edges, 
because the convergence is not uniform at the vertices.
However, we can derive the integrated version of the eigenfunction asymptotic from Proposition~\ref{smallt.prop}.
\begin{proposition}\label{htx.asymp.prop}
Suppose the edge ${\bf e}$ is parametrized by $x \in [0,a]$, and denote by $h_{\bf e}(t,\cdot)$ the restriction of $h(t,\cdot)$ to ${\bf e}$.
For $f \in C[0,a]$, as $t \to 0$,
\[
\int_0^a f(x) h_{\bf e}(t,x)\>dx = \frac{1}{\sqrt{4\pi t}} \int_0^a f(x)\>dx + \left(\frac{2}{d_0}-1\right) \frac{f(0)}{4} +
\left(\frac{2}{d_1}-1\right)  \frac{f(a)}{4} + o(t),
\]
where $d_0$ and $d_1$ are the degrees of the vertices at $x=0$ and $a$, respectively.
\end{proposition}
\begin{proof}
The leading term is clear from \eqref{Kqq.smallt}, and the contribution from the remainder $R(t,x)$ is $O(e^{-a_0^2/t})$.
The vertex terms are easily calculated from
\[
\lim_{t\to 0} \frac{1}{\sqrt{4\pi t}} \int_0^a f(x) e^{-x^2/t}\>dx = \frac{f(0)}4.
\]
\end{proof}

Applying the Karamata theorem as above gives the following asymptotic for the average distribution of eigenfunctions:
\begin{corollary}\label{eig.dist.cor}
For an edge ${\bf e}$ parametrized by $x \in [0,a]$, let $\psi_j|_{\bf e}(x)$ denote the restriction of the $j$th eigenfunction.  
For a continuous function $f$ on $[0,a]$,
\[
\lim_{N\to \infty} \frac{1}{N} \sum_{j=1}^N \int_0^a f(x) \psi_j|_{\bf e}(x)^2\>dx = \frac{1}{L} \int_0^a f(x)\>dx.
\]
\end{corollary}

The corresponding result is well known in the case of Riemannian manifolds, with no ergodicity assumptions \cite[\S4]{CdV:1985}.
The stronger condition of quantum ergodicity, i.e., the convergence of $\psi_j(x)^2\>dx$ to uniform measure 
for sequences of eigenfunctions, is known to fail for quantum graphs in general \cite{BKW:2004, KMW:2003}.

The result of Corollary~\ref{eig.dist.cor} could also be interpreted in terms of the amplitudes of eigenfunctions.  Within a given edge 
each (normalized) eigenvalue takes the form
\begin{equation}\label{psi.amp}
\psi_j|_{\bf e}(x) = b_j({\bf e}) \cos(\sigma_j x + \phi_j),
\end{equation}
where $\lambda_j = \sigma_j^2$.  We will assume that the phase is chosen that $b_j({\bf e})>0$,
which fixes $b_j({\bf e})$ uniquely (depending on the choice of basis $\{\psi_j\}$).  

The semiclassical limit measures associated to sequences of eigenfunctions
were studied by Colin de Verdi\`ere \cite{CdV:2015}, who noted that a sequence of measures $\psi_j^2 \>dx$ has a weak
limit if and only if the  corresponding sequence $b_j({\bf e})^2$ is convergent.  This follows from the identity
\[
\psi_j|_{\bf e}(x)^2 =  \frac{b_j({\bf e})^2}{2} \Bigl( 1 + \cos(2\sigma_j x + 2\phi_j) \Bigr),
\]
and the Riemann-Lebesgue lemma.
The same reasoning applies to the average measure, so that Corollary~\ref{eig.dist.cor} is equivalent to the
statement that
\begin{equation}\label{bj.avg}
\lim_{N\to \infty} \frac{1}{N} \sum_{j=1}^N b_j({\bf e})^2 = \frac{2}{L}.
\end{equation}

\section{Edge trace formula}

In Roth's statement of the heat trace formula \cite{Roth:1984}, 
a \emph{cycle} is defined as an equivalence class
of closed, oriented paths on $G$, excluding trivial paths.  
A path is closed if its initial and final vertices match, and two paths are equivalent if
they are related by cyclic permutation of the bonds.  A cycle is
\emph{primitive} if it cannot be written as the iteration of
a smaller cycle.  Let $\calC$ denote the set of primitive cycles of $G$.  

\begin{theorem}[Roth]
For a compact metric graph $G$ with $V$ vertices and $E$ edges, 
\[
\sum_{j=1}^\infty e^{-t\lambda_j} =  \frac{L}{\sqrt{4\pi t}} + \frac12(V-E) + \frac{1}{\sqrt{4\pi t}} \sum_{\gamma \in \calC} \sum_{k=1}^\infty
\alpha(\gamma) \ell(\gamma) e^{-k^2\ell(\gamma)^2/4t}.
\]
\end{theorem}

Using the path sum formula for the heat kernel, we can develop a variant of the trace formula which is localized 
to a specific edge.  
As in the previous section, we focus on a single edge ${\bf e}$ of length $a$, with eigenfunctions
represented by \eqref{psi.amp} as functions of $x \in [0,a]$.  In particular, the restrictions of the eigenfunctions
to ${\bf e}$ determine a sequence of amplitudes $b_j({\bf e})>0$.

Let $\calP_{\bf e}$ denote the set of closed paths which begin with a bond contained in the edge ${\bf e}$.  
The trivial path is included.

\begin{theorem}[edge trace formula]\label{edge.tr.thm}
For a fixed edge ${\bf e}$, define the eigenfunction amplitudes $b_j({\bf e})$ as in \eqref{psi.amp}.  
There exists a constant $c_{\bf e}$ such that for $t>0$,
\[
\frac12 \sum_{j=1}^\infty e^{-\lambda_j t} b_j({\bf e})^2 =  \frac{1}{\sqrt{4\pi t}} \sum_{\gamma \in \calP_{\bf e}} 
\alpha(\gamma) e^{-\ell(\gamma)^2/4t} + c_{\bf e}.
\]
\end{theorem}
\begin{proof}
Let $q$ denote an interior point in ${\bf e}$ corresponding to $x \in (0,a)$.  In the expansion formula \eqref{KPS.formula} for $h(t,q)$,
we can subdivide 
\[
\calP(q,q) = \calP_0(q,q) \cup \calP_1(q,q),
\]
where $\calP_0(q,q)$ consists of the trivial path along with 
the paths which return to $q$ from the opposite of the initial direction.  The set $\calP_1(q,q)$
consists of paths whose final bond is the reverse of the initial. These class are illustrated in Figure~\ref{edge_type.fig},
for a path of the form $(q,\vec{e}_1,\dots, \vec{e}_n,q)$.

Note that the length of a path in $\calP_0(q,q)$ is independent of $x$.  
There is a length-preserving bijection from $\calP_0(q,q)$ to $\calP_{\bf e}$
given by 
\[
(q,\vec{e}_1,\dots,\vec{e}_n,q) \mapsto (v_0,\vec{e}_n+\vec{e}_1,\vec{e}_2,\dots,\vec{e}_{n-1},v_0),
\]
where $v_0 := \del^-(\vec{e}_n)$.
The artificial vertex $q$ is deleted in this process, but since $q$ has degree 2, 
the coefficient $\alpha$ is preserved.  The contribution to the path sum formula for $h_{\bf e}(t,\cdot)$ 
from paths in $\calP_0(q,q)$ can thus be written as
\[
h_0(t) := \frac{1}{\sqrt{4\pi t}} \sum_{\gamma \in \calP_{\bf e}} \alpha(\gamma) e^{-\ell(\gamma)^2/4t}.
\]

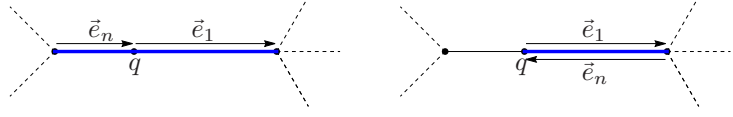
\begin{figure}
\begin{tabular}{c@{\hspace{.3in}}c}
\begin{overpic}[scale=.5]{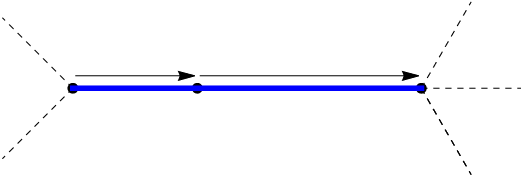}
\put(36,11){$q$}
\put(55,21){$\vec{e}_1$}
\put(24,21){$\vec{e}_n$}
\end{overpic} &
\begin{overpic}[scale=.5]{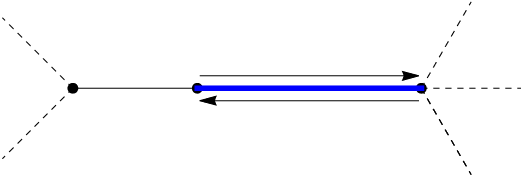}
\put(35,11){$q$}
\put(55,21){$\vec{e}_1$}
\put(55,8){$\vec{e}_n$}
\end{overpic}
\end{tabular}
\caption{For a non-trivial path in $\calP_0(q,q)$, the initial and final bonds are arranged as shown on the left. 
For paths in $\calP_1(q,q)$, the final bond is the reverse of the initial, as shown on the right.}\label{edge_type.fig}
\end{figure}

For $\gamma \in \calP_1(q,q)$, the length takes the form
\[
\ell(\gamma) = \pm(2x + c_\gamma),
\]
for some $c_\gamma \in \bbR$.  
The contribution to $h_{\bf e}(t,\cdot)$ from paths in $\calP_1(q,q)$ thus takes the form
\[
h_1(t,x) := \frac{1}{\sqrt{4\pi t}} \sum_{\gamma \in \calP_1(q,q)} \alpha(\gamma) e^{-(2x+c_\gamma)^2/4t}.
\]
Note that the terms in the series expansion for $h_1$ individually satisfy the equation
\[
(\partial_t - \tfrac14 \partial^2_x) \frac{1}{\sqrt{4\pi t}} e^{-(2x+c_\gamma)^2/4t} = 0.
\]
The sum over $\calP_1(q,q)$ is uniformly convergent on $[\epsilon,\infty)\times G$, for every $\epsilon>0$, 
by the estimates used in the proof of Proposition~\ref{KPS.prop},
and the same is true of the series for its derivatives.
Therefore, for all $t>0$ the term $h_1$ satisfies a modified heat equation:
\begin{equation}\label{h1.eq}
(\partial_t - \tfrac14 \partial^2_x) h_1 = 0.
\end{equation}

With $h_0$ and $h_1$ defined as above,
\begin{equation}\label{htx.path}
h_{\bf e}(t,x) = h_0(t) + h_1(t,x).
\end{equation}
On the other hand, in the notation \eqref{psi.amp}, the eigenfunction expansion gives
\begin{equation}\label{htx.eigf}
h_{\bf e}(t,x) = \frac12 \sum_{j=1}^\infty b_j({\bf e})^2 e^{-\lambda_j t}  \Bigl( 1 + \cos(2(\sigma_j x + \phi_j)) \Bigr).
\end{equation}
This expression admits a decomposition 
\[
h_{\bf e}(t,x) = \eta_0(t) + \eta_1(t,x),
\]
where
\[
\eta_0(t) :=  \frac12 \sum_{j=1}^\infty b_j({\bf e})^2 e^{-\lambda_j t}
\]
and 
\[
\eta_1(t,x) := \frac12 \sum_{j=1}^\infty b_j({\bf e})^2 e^{-\lambda_j t}  \cos(2(\sigma_j x + \phi_j)).
\]
By \eqref{bj.avg} and the Weyl law for $\lambda_j$, the series for $\eta_1$ converges uniformly
$[\epsilon,\infty)\times G$, for all $\epsilon>0$, as do the series for its derivatives.  Therefore $\eta_1$ also satisfies the
modified heat equation,  
\begin{equation}\label{eta1.eq}
(\partial_t - \tfrac14 \partial^2_x) \eta_1 = 0,
\end{equation}
for $t>0$.
From the equations \eqref{h1.eq} and \eqref{eta1.eq} we deduce that
\[
\del_t \left[h_0(t) - \eta_0(t)\right] = 0,
\]
which completes the proof.
\end{proof}

\section{Equation for the kernel on the diagonal}

In this section we will study the restriction of the heat kernel to the diagonal on a single edge of $G$.
As before, we assume that $G$ is a compact metric graph, with standard Kirchhoff-Neumann conditions at the vertices.
 
Assume that the edge ${\bf e}$ is parametrized by $x \in [0,a]$,
and denote the restriction of $h(t,\cdot)$ by $h_{\bf e}(t,\cdot)$.  
A useful property of $h_{\bf e}(t,\cdot)$, which we can see implicitly in the proof of Theorem~\ref{edge.tr.thm}, is that 
its spatial derivative satisfies a heat equation of its own.

\begin{prop}\label{PDEprop1}
On each edge ${\bf e}$, $\del_xh_{\bf e}$ satisfies a modified heat equation,
\begin{equation}\label{heat14}
(\partial_t - \tfrac14 \partial^2_x) \partial_x h_{\bf e}(t,x) = 0.
\end{equation}
As a function on $G$, $h(t,\cdot)$ is continuous for $t>0$ and satisfies Kirchoff-Neumann conditions at the vertices.
\end{prop}

\begin{proof}
Consider the local expansion for $h_{\bf e}(t,\cdot)$ given by \eqref{htx.eigf}, using the notation \eqref{psi.amp}
for the restriction of eigenfunctions to the edge.  
As noted in the proof of Theorem~\ref{edge.tr.thm},
this series converges uniformly, and the same is true for the series for its derivatives 
obtained by term-by-term differentiation.  We thus calculate
\begin{equation}\label{hkrep}
(\partial_t - \tfrac14 \partial^2_x)h_{\bf e}(t,x) = \tfrac12 \sum_{j} b_j({\bf e})^2 e^{- \lambda_j t}  \lambda_j.
\end{equation}
The right side is independent of $x$, so this proves \eqref{heat14}.

As to the vertex conditions,
the continuity of $h(t,\cdot)$ at vertices is clear from \eqref{hk.eigf}, since the eigenfunctions are continuous.
Differentiating \eqref{hk.eigf} directly gives
\begin{equation*}
\partial_x h_{\bf e}(t,x) = -2 \sum_{j} b_j({\bf e})^2 e^{- \lambda_j t} \sigma_j \cos(\sigma_j x + \phi_j)\sin(\sigma_j x + \phi_j).
\end{equation*}
Evaluating at $x=0$ gives
\[
\partial_x h_{\bf e}(t,x) \big|_{x=0} = 2 \sum_{j} b_j({\bf e})^2 e^{- \lambda_j t} \psi_j(v) \del_{\bf e}\psi_j(v),
\]
where $v$ denotes the vertex at $x=0$, and $\del_{\bf e}$ is the outward derivative in direction of edge ${\bf e}$.
Since $\psi_j(v)$ is independent of ${\bf e}$,
when this expression is summed over the edges incident to $v$, the result is $0$ by the Kirchoff-Neumann conditions
on $\psi_j$.
\end{proof}

At a given vertex, we can arrange the parametrizations of outgoing edges so that $\del_x h$ satisfies 
the anti-Kirchoff-Neumann vertex condition, as described in \cite[Def.~6]{FKW:2007}. However, 
$\del_x h$ is not globally well-defined because of the dependence on edge orientation.  
On a particular edge, the difficulty in solving the modified heat equation \eqref{heat14} lies in fact that the boundary conditions on 
$\del_x h_{\bf e}$ at the endpoints of ${\bf e}$ are generally time-dependent.  Moreover, there is no apparent way
to write the boundary values explicitly in the general case. The inhomogeneous equation \eqref{hkrep} has the same problem;
the source term is generally not calculable.

We can at least work out the initial condition for $\del_x h_{\bf e}$ at $t=0$, from the following:
\begin{proposition}\label{dhtx.asymp.prop}
For $f \in C^1[0,a]$, as $t \to 0$,
\[
\begin{split}
\int_0^a f(x) \del_x h_{\bf e}(t,x)dx &= \frac{1}{\sqrt{4\pi t}} \left[ \left(\frac{2}{d_1} - 1\right) f(a) - \left(\frac{2}{d_0} - 1\right) f(0) \right] \\
&\qquad + \frac14 \left(\frac{2}{d_0} - 1\right) f'(0) + \frac14 \left(\frac{2}{d_1} - 1\right) f'(a) + o(t).
\end{split}
\]
\end{proposition}

\begin{proof}
Integration by parts gives
\[
\int_0^a f(x) \del_x h_{\bf e}(t,x)dx = f(\cdot) h_{\bf e}(t,\cdot)\Big|^a_0 - \int_0^a f'(x) h_{\bf e}(t,x)\>dx.
\]
By Proposition~\ref{smallt.prop} the first term has the asymptotic
\[
f(\cdot) h_{\bf e}(t,\cdot)\Big|^a_0 = \frac{1}{\sqrt{4\pi t}} \left[ \frac{2}{d_1}f(a) - \frac{2}{d_0}f(0) \right] + O(t^\infty).
\]
For the second term, we apply Proposition~\ref{htx.asymp.prop} to deduce that
\[
\int_0^a f'(x) h_{\bf e}(t,x)\>dx =  \frac{1}{\sqrt{4\pi t}} [f(a) - f(0)] + \frac14 \left(\frac{2}{d_0} - 1\right) f'(0) + \frac14 \left(\frac{2}{d_1} - 1\right) f'(a) + o(t).
\]
\end{proof}

For certain graphs, such as equilateral complete graphs, star graphs, pumpkin graphs, etc., we can deduce that
$\del_xh_{\bf e}(t,\cdot)$ will vanish at the vertices.  In these cases, the equation \eqref{heat14} can be solved using Proposition~\ref{dhtx.asymp.prop}.
\begin{proposition}\label{dhtx.sin.prop}
For an edge ${\bf e}$ parametrized by by $x \in [0,a]$, suppose that 
\[
\del_xh_{\bf e}(t,0) = \del_xh_{\bf e}(t,a) = 0.
\]
Then $\del_xh_{\bf e}(t,\cdot)$ admits the expansion, for $t>0$,
\begin{equation}\label{dhtx.exp}
\del_xh_{\bf e}(t,x) = \sum_{n=1}^\infty c_n e^{-(\pi n/2a)^2 t} \sin\left(\tfrac{\pi n x}{a} \right),
\end{equation}
where 
\[
c_n = \frac{\pi n}{2a^2} \left[\left(1 - \frac{2}{d_0} \right) + (-1)^n  \left(1 - \frac{2}{d_1} \right)\right]
\]
\end{proposition}
\begin{proof}
Existence of the expansion \eqref{dhtx.exp} follows from the equation \eqref{heat14} and the Dirichlet boundary conditions assumed for $\del_xh_{\bf e}(t,\cdot)$.
By the Fourier series coefficient formula,
\[
c_n e^{-(\pi n/2a)^2 t} = \frac{2}{a} \int_0^a \del_xh_{\bf e}(t,x) \sin\left(\tfrac{\pi n x}{a} \right) dx. 
\]
Applying Proposition~\ref{dhtx.asymp.prop} and taking $t \to 0$ then gives the stated formula for $c_n$.
\end{proof}

\section{Applications to symmetric graphs}

As above, we continue to assume Kirchhoff-Neumann conditions on a compact graph $G$.
In this section we assume that $G$ has a high degree of symmetry and work out the consequences for the heat kernel.
\begin{defi}
A metric graph $G$ is said to be {\em symmetric about a vertex} $v$ if all edges incident to $v$ have the same 
length and the graph is invariant under any permutation of those edges. 
\end{defi}

We first check that symmetry about a vertex will ensure that the hypothesis of Proposition~\ref{dhtx.sin.prop} is satisfied at that vertex.  
\begin{lemma}\label{sym.htx.lemma}
The derivative of the diagonal heat kernel $h_{\bf e}(t,\cdot)$ vanishes at each vertex about which $G$ is symmetric.
\end{lemma}
\begin{proof}
As noted in Proposition~\ref{PDEprop1},  $h_{\bf e}(t,\cdot)$ satisfies Kirchhoff-Neumann vertex conditions.
Since the heat kernel is unique, symmetry about a vertex $v$ guarantees that value of its derivative at $v$ on any incident edge is the same.  
By the Kirchhoff-Neumann vertex conditions, that value must be 0.
\end{proof}

An equilateral star graph is symmetric about each vertex, and so
Proposition~\ref{dhtx.sin.prop} applies to the star graph from Example~\ref{star.ex}.  
Setting $d_0 = d$, $d_1 =1$ in \eqref{dhtx.exp} gives
\[
\del_xh_{\bf e}(t,x) = \frac{\pi (d-1)}{da^2} \sum_{n\text{ odd}} n e^{-(\pi n/2a)^2 t} \sin\left(\tfrac{\pi n x}{a} \right)
- \frac{\pi}{da^2} \sum_{n\text{ even}} n e^{-(\pi n/2a)^2 t} \sin\left(\tfrac{\pi n x}{a} \right),
\]
in agreement with \eqref{htx.star}.

Symmetry about each vertex implies in particular that the discrete graph associated to $G$ is \emph{locally arc-transitive}.
For simplicity, in the remainder of this section we restrict our attention to cases where $G$ is regular.  
Examples include complete graphs, cubes and hypercubes, flowers, pumpkins, periodic pumpkin chains, 
and (symmetric) torus grid graphs. 
A regular, locally arc-transitive graph need not be vertex-transitive, as illustrated by the Folkman graph shown 
in Figure~\ref{folkman.fig}.

\begin{figure}
\begin{center}
\includegraphics[scale=.5]{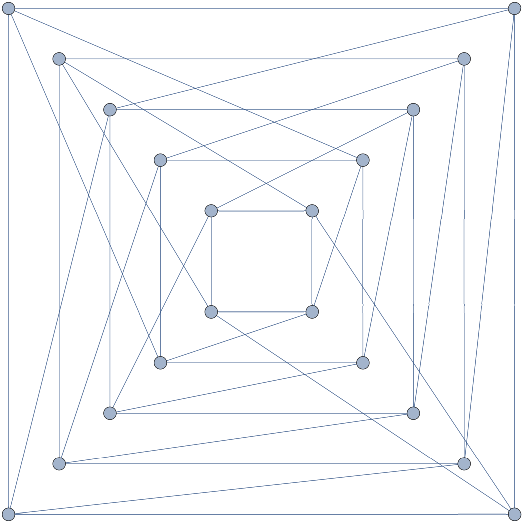}
\end{center}
\caption{The Folkman graph is locally arc-transitive, but not vertex-transitive.}\label{folkman.fig}
\end{figure}

The spectrum of a general equilateral graph was worked out by von Below and Mugnolo \cite{vBM:2013}.  
If $a$ denotes the common edge length, then the eigenvalues not contained in 
$(\pi/a)\bbZ$ are associated to eigenvalues of the discrete Laplacian of $G$.  
For a regular graph the discrete Laplacian is equivalent to the adjacency matrix,
\[
A_{ij} := \begin{cases} 1, & v_i \text{ and }v_j\text{ are connected by an edge}, \\
0,& \text{otherwise}. \end{cases}
\]

\begin{theorem}\label{sym.htx.thm}
Let $G$ be a compact, regular, equilateral metric graph which is symmetric about each vertex. If $G$ has 
edge length $a$, vertex degree $d$, and $n$ vertices, then on each edge,
\[
\begin{split}
h_{\bf e}(t,x) &= \frac{2}{dna} \sum_{k\in \bbZ} e^{-(c\pi k/a)^2t} 
+ \frac{d-2}{2da} \sum_{k \in \bbZ} e^{-(\pi k/a)^2t} \left[ 1 - \cos \left( \tfrac{2\pi k x}{a} \right) \right]\\
&\qquad +  \frac{2}{dna} \sum_{\sigma \in Q} \sum_{k \in \bbZ} \mu_\sigma e^{-(\sigma + 2\pi k/a)^2t},
\end{split}
\]
where $c=1$ if $G$ is bipartite and $2$ if not, and 
\[
Q := \Bigl\{\sigma \in (0,\pi): \>d \cos(\sigma a) \in \operatorname{Spec(A)},\> \sigma \notin (\pi/a)\bbN \Bigr\},
\]
with $\mu_\sigma$ the multiplicity of $d \cos(\sigma a)$ as an eigenvalue of $A$.
\end{theorem}
\begin{proof}
By Lemma~\ref{sym.htx.lemma}, on each edge we have
\begin{equation}\label{hte,start}
h_{\bf e}(t,x) = h_0(t) - \frac{d-2}{da} \sum_{k=1}^\infty e^{-(\pi k/a)^2t}  \cos\left( \tfrac{2\pi k x}{a} \right), 
\end{equation}
for some function $h_0(t)$.  Integrating over $x$ gives the heat trace
\[
\sum_{j=1}^\infty e^{-\lambda_jt} = \int_G h_{\bf e}(t,x) \>dx = \frac{dna}{2} h_0(t).
\]

To write the heat trace explicitly, we recall the spectrum from \cite[Thm.~3.2]{vBM:2013}.  
The eigenvalues $\lambda_j$ are as follows:
\begin{enumerate}
\item  $\lambda = (\sigma + 2\pi k/a)^2$ for $\sigma \in Q$ with multiplicity $\mu_q$.
\item  $\lambda = k^2 \pi^2/a^2$ for $k \in \bbN_0$, with multiplicities $m(0) = 1$ and 
\[
m(k^2 \pi^2/a^2) = \begin{cases} (d/2-1)n, & k \text{ odd},\\
(d/2-1)n + 2, & k \text{ even}.  \end{cases}
\]
for $k \ge 1$ if $G$ is not bipartite.  If $G$ is bipartite then $m(k^2 \pi^2/a^2)  = (d/2-1)n+2$ for all $k \ge 1$.
\end{enumerate}
The multiplicity $(d/2-1)n$ combines with the factor $2/dna$ to give the same coefficient $(d-2)/da$ as the cosine term
in \eqref{hte,start}. This leaves eigenvalues from (2) with multiplicity 2 for either $\lambda \in (\pi \bbN/a)^2$ or
$\lambda \in (2\pi \bbN/a)^2$, depending on whether $G$ is bipartitie.
\end{proof}

The behavior of $h_{\bf e}(t,\cdot)$ near the vertices is illustrated in Figure~\ref{complete_htx.fig}.
Note that the non-constant term in the formula $h_{\bf e}(t,\cdot)$ depends only on $d$ and $a$ and is otherwise 
independent of the graph.  This term can be expressed in terms of the Jacobi theta function,
\[
\vartheta_3(z; \tau) := \sum_{k\in\bbZ} e^{i \pi \tau k^2} \cos(2kz).
\]
In particular,
\[
\sum_{k\in\bbZ} e^{-(\pi k/a)^2t} \cos\left( \tfrac{2\pi k x}{a} \right) = \vartheta_3\left( \tfrac{\pi x}{a}; \tfrac{i \pi t}{a^2} \right).
\]

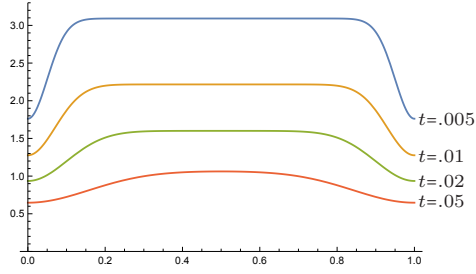
\begin{figure}
\begin{center}
\begin{overpic}[scale=.6]{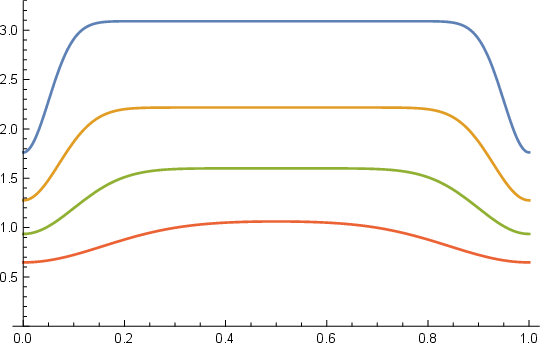}
\put(99,34){$\scriptstyle t = .005$}
\put(99,25){$\scriptstyle t = .01$}
\put(99,19){$\scriptstyle t = .02$}
\put(99,14){$\scriptstyle t = .05$}
\end{overpic}
\end{center}
\caption{The function $h_{\bf e}(t,\cdot)$ on an edge of a symmetric tetrahedron with $a=1$.}\label{complete_htx.fig}
\end{figure}

The expressions in Theorem \ref{sym.htx.thm} can be written in the form of Proposition~\ref{KPS.prop} using
the Poisson summation formula,
\[
\begin{split}
\sum_{k \in \bbZ} e^{-(c\pi k/a)^2t} &= \frac{2a/c}{\sqrt{4\pi t}} \sum_{l\in\bbZ} e^{-(la)^2/c^2t}, \\
\sum_{k \in \bbZ} e^{-(\pi k/a)^2t} \left[ 1 - \cos \left( \tfrac{2\pi k x}{a} \right) \right]
&= \frac{2a}{\sqrt{4\pi t}} \sum_{l \in \bbZ} \left( e^{-(la)^2/t} - e^{-(la-x)^2/t} \right), \\
\sum_{k \in \bbZ} e^{-(\sigma + 2\pi k/a)^2t} &= \frac{a}{\sqrt{4\pi t}} \sum_{l\in\bbZ} \cos(la\sigma) e^{-(la)^2/4t}.
\end{split}
\]
This yields the expansion
\[
\begin{split}
h_{\bf e}(t,x) &= \frac{1}{\sqrt{4\pi t}} \biggl[ 
\frac{4}{cdn} \sum_{l\in\bbZ} e^{-(la)^2/c^2t}
+ \frac{d-2}{d} \sum_{l \in \bbZ} \left( e^{-(la)^2/t} - e^{-(la-x)^2/t} \right) \\
&\qquad\qquad +\frac{2}{dn} \sum_{\sigma \in Q} \sum_{l\in\bbZ} \mu_\sigma \cos(la\sigma) e^{-(la)^2/4t}
\biggr].
\end{split}
\]
Note that the coefficients of $1/\sqrt{4\pi t}$ add up to 1, as required by Proposition~\ref{KPS.prop}, 
since $Q$ has $n-2/c$ elements, counting multiplicities.  
Organizing the sum by lengths gives the following:
\begin{corollary}
For $G$ as in Theorem \ref{sym.htx.thm}, 
\[
h_{\bf e}(t,x) = \frac{1}{\sqrt{4\pi t}} \biggl[ 1 + \sum_{l=2}^\infty A_l e^{-(la)^2/4t} - \frac{d-2}{d} \sum_{l\in \bbZ} e^{-(la-x)^2/t} \biggr],
\]
where, if $G$ is not bipartite,
\[
A_l = \begin{dcases}\frac{4}{dn} \Bigl[1+ \sum_{\sigma \in Q} \mu_\sigma \cos(la\sigma)\Bigr] , &l \text{ odd}, \\
\frac{4}{dn} \Bigl[1+ \sum_{\sigma \in Q} \mu_\sigma \cos(la\sigma)\Bigr] +\frac{2(d-2)}{d},&l\text{ even}, \end{dcases}
\]
and, if $G$ is bipartite,
\[
A_l = \begin{dcases}\frac{4}{dn} \sum_{\sigma \in Q} \mu_\sigma \cos(la\sigma) , &l \text{ odd}, \\
\frac{4}{dn} \Bigl[2+ \sum_{\sigma \in Q}  \mu_\sigma \cos(la\sigma)\Bigr] +\frac{2(d-2)}{d},&l\text{ even}, \end{dcases}
\]
\end{corollary}

\begin{example}
Let $G$ be the complete regular graph with $n = d+1$ vertices of degree $d$ and uniform 
edge length $a$.  The adjacency matrix $A$ consists of zeros on the diagonal and ones off-diagonal.
In this case $Q$ consists of a single value $\sigma = \arccos(-1/d)$, with multiplicity $\mu_\sigma = d$.  
\end{example}



\end{document}